\documentclass[11pt, reqno]{amsart}
\usepackage[a4paper, left=1in, right=1in, top=1in, bottom=1in]{geometry}
\usepackage{amsmath}
\usepackage{amssymb}
\usepackage[T1]{fontenc}
\usepackage{newtxtext}
\usepackage{newtxmath}

\usepackage{hyperref}
\hypersetup{colorlinks,
linkcolor={black},
filecolor={black},
citecolor={blue},
urlcolor={blue}}  

\usepackage{amsthm}
\newtheorem{thm}{Theorem}[section]
\newtheorem{cor}{Corollary}[section]
\newtheorem{lma}{Lemma}[section]
\numberwithin{equation}{section}
\begin{document}
 
\title{Arithmetic properties of $(\ell,m)$-regular colored partitions}
\author{Yashas N., C. Shivashankar and S. Chandankumar}

\address[Yashas N.]{Department of Mathematics and Statistics, M. S. Ramaiah University of Applied Sciences, Peenya Campus, Peenya 4th Phase, Bengaluru-560 058, Karnataka, India.} 
\email{yn37@proton.me}

\address[C. Shivashankar]{Department of Mathematics, Vidyavardhaka College of Engineering, Gokulam III Stage, Mysuru, Karnataka 570 002, India} \email{shankars224@gmail.com}

\address[S. Chandankumar]{Department of Mathematics and Statistics, M. S. Ramaiah University of Applied Sciences, Peenya Campus, Peenya 4th Phase, Bengaluru-560 058, Karnataka, India.} 
\email{chandan.s17@gmail.com}

\subjclass[2010]{11P83, 05A15, 05A17} \keywords{$(\ell,m)-$regular partitions,color partitions}

\maketitle
\begin{abstract}
    Let $b^{k}_{\ell,m}(n)$ denotes the number of $k-$colored partitions of $n$ into parts that are not multiples of $\ell$ or $m$. We establish several congruence relations for $b_{\ell,m}(n)$. For instance, for any nonnegative integer $n$
    $$b^{2}_{4,5}(8n+7) \equiv 0 \pmod{40}.$$
\end{abstract}
\section{Introduction}
\noindent An unrestricted partition of a non-negative integer $n$ is defined as a sum of a finite non-increasing sequence of positive integers. For example, there are $5$ partitions of $4$:
$$4,\,\,\,3+1,\,\,\,2+2,\,\,\,2+1+1,\,\,\,1+1+1.$$
We denote the number of unrestricted partition of $n$ by $p(n)$. The generating function of $p(n)$ is given by 
$$\sum_{n=0}^{\infty} p(n) q^n =\dfrac{1}{(q;q)_\infty},$$ 
where
\[(a;q)_n=\begin{cases}
	1, & \text{for} \ \ n=0,\\
	\displaystyle \prod_{k=1}^{n} (1-aq^{k-1}), &\text{for} \ \ n>0,
\end{cases}\]
is the $q$-shifted factorial and, 
\[(a;q)_\infty=\lim_{n\rightarrow\infty}(a;q)_n, \qquad |q|<1.\]

An $\ell$-regular partition is a partition in which none of the parts is divisible by $\ell$. Let $b_{\ell}(n)$ denote the number of $\ell$-regular partitions of $n$, with the notation that $b_{\ell}(0) = 1$. The generating function for $b_{\ell}(n)$ is given by:

\begin{equation}\label{bkln}
	\sum_{n=0}^{\infty} b_{\ell}(n) q^n = \frac{f_{\ell}}{f_1},
\end{equation}
where $\displaystyle f_r=(q^r;q^r)_\infty=\Pi_{n=1}^\infty (1-q^{nr})$ for $|q|\leq 1.$ These notations are followed through out this article. Let $b^{k}_{\ell}(n)$ denotes the number of $\ell-$ regular partitions of $n$ in $k$ colors then its generating function is given by 

\begin{equation} \label{1main}
	\displaystyle \sum_{n=0}^\infty b^{k}_{\ell}(n)q^n = \frac{f_{\ell}^k }{f_1^k}.
\end{equation}

Similarly, for two coprime positive integers $\ell$ and $m$, a $(\ell, m)$-regular partition of $n$ is a partition in which none of the parts is divisible by $\ell$ or $m$. Let $b_{\ell,m}(n)$ denote the number of such partitions of $n$, with the convention that $b_{\ell,m}(0) = 1$. The generating function for $b_{\ell,m}(n)$ is given by

\begin{equation}
	\sum_{n=0}^{\infty} b_{\ell,m}(n) q^n = \frac{f_{\ell} f_{m}}{f_1 f_{\ell m}}.
\end{equation}

In recent years, the arithmetic properties of $b_{\ell,m}(n)$ have been extensively studied. In 2016, M. S. Mahadeva Naika, B. Hemanthkumar, and H. S. Sumanth Bharadwaj \cite{naika2016congruences} established infinite families of congruences modulo 2 for $b_{3,5}(n)$. Drema and Saikia \cite{drema2022arithmetic} derived similar congruences modulo 2 for $b_{4,m}(n)$ when $m = 3, 5, 7,$ and $9$. 

Further contributions were made by T. Kathiravan, K. Srinivas, and Usha K. Sangale \cite{kathiravan2023some}, who derived infinite families of congruences for $b_{\ell,m}(n)$ modulo 2 when $(\ell, m) =(3, 8),(4, 7),$ for $b_{\ell,m}(n)$ modulo 8, modulo 9 and modulo 12 when $(\ell,m) = (4, 9).$  Recently, Kathiravan, Dipramit Majumdar, and Usha \cite{kathiravan2023some1}, utilizing theta function identities and Newman's results, obtained infinite families of congruences modulo 2 for $b_{\ell,m}(n)$ when $(\ell, m) = (2, 7), (5, 8), (4, 11)$, and modulo 4 for $(4,5)$-regular partitions.

In this article,  we study the partition function $b^{k}_{\ell,m}(n)$, $k$-colored partition of $n$ in which no part is a multiple of $\ell$ or $m$. The generating function of such partition function is given by
\begin{equation} \label{main}  
\sum_{n=0}^\infty b^{k}_{\ell,m}(n) q^n = \frac{f_{\ell}^k f_m^k}{f_1^k f_{lm}^k}.  
\end{equation}  

Mahadeva Naika and Harishkumar T. \cite{naika2019congruences} established numerous infinite families of congruences modulo powers of 2 for $b^2_{4,5}(n)$. Recently, Veena and Fathima \cite{veena2023fathima}, employing the theory of the modular forms and $q$-series techniques, derived several infinite families of congruences modulo 4 for $b^2_{4,9}(n)$.  

The Legendre symbol is a function of $a$ and $p\geq3$, and is used frequently in this article. It is defined by

\[
\left( \frac{a}{p} \right) =
\begin{cases} 
1 & \text{if } a \text{ is a quadratic residue modulo } p \text{ and } a \not\equiv 0 \pmod{p}, \\
-1 & \text{if } a \text{ is a quadratic non-residue modulo } p, \\
0 & \text{if } a \equiv 0 \pmod{p}.
\end{cases}
\]
  In Section \ref{s3}, we study the partition functions $b^2_{3,2t}(n)$ and $b^2_{3,4t}(n)$ and derive a characterization for $b^2_{3,4}(4n+1)$. In Section \ref{s4}, we establish an infinite family of congruence modulo 4 for $b^2_{2,5}(n).$ Section \ref{s5} is devoted the partition function $b^2_{5,4t}(n)$. In the last section, we study the arithmetic properties of the partition function $b^3_{2,3t}(n).$
	\section{Preliminaries}
	 To begin this section, we introduce Ramanujan’s general theta function
	$f(a,b),$ defined as :
\begin{align}
	f(a,b) = \sum_{n=-\infty}^{\infty} a^{\frac{n(n+1)}{2}} b^{\frac{n(n-1)}{2}}, \quad |ab| < 1.
\end{align}
The following definitions of theta functions  $\varphi$, $\psi$ and $f$ with $|q|<1$ are classical:
\begin{eqnarray}
	\varphi(q)&:=&f(q,q)=\sum_{n=-\infty}^{\infty}q^{n^2} = \frac{f_2}{f_1^2f_4^2},\\
	\psi(q)&:=&f(q,q^3) =\dfrac{1}{2}\sum_{n=-\infty}^{\infty}q^{n(n+1)/2} =\frac{f_2^2}{f_1},\\
    \text{and}\\
	f(-q)&:=&f(-q,-q^2)=\sum_{n=-\infty}^{\infty}(-1)^n q^{n(3n-1)/2}=f_1.
\end{eqnarray}
\begin{lma} \cite[p. 40, Entry 25]{berndt2012ramanujan} The following 2-dissections holds 
	\begin{equation}\label{f1.2}
		f_1^2 =\frac{f_2 f_8^5}{f_4^2 f_{16}^2}
		-2q\frac{f_2 f_{16}^2}{f_8},
	\end{equation}
	\begin{equation}
		\label{1_f1.2}
		\frac{1}{f_1^2} =\frac{f_8^5}{f_2^5 f_{16}^2}
		+2q\frac{f_4^2 f_{16}^2}{f_2^5 f_8},
	\end{equation}
	\begin{equation}
		\label{1_f1.4}
		\frac{1}{f_1^4} =\frac{f_4^{14}}{f_2^{14} f_{8}^4}
		+4q\frac{f_4^2 f_{8}^4}{f_2^{10}}.
	\end{equation}
\end{lma}

\begin{lma} In \cite{xia2013analogues}, the following 2-dissections holds
	\begin{align}
		\frac{f_3}{f_1} &=
		\frac{f_4 f_6 f_{16}f_{24}^{2}}{f_2^{2} f_8 f_{12}f_{48}}
		+ q \frac{f_6 f_8^2 f_{48}}{f_2^2f_{16}f_{24}},\label{f3f1}\\
\frac{f_3^2}{f_1^2}
&=\frac{f_4^4 f_6 f_{12}^2}{f_2^5 f_8 f_{24}}
+2q\frac{f_4 f_6^2 f_8 f_{24}}{f_2^4 f_{12}},\label{f3.2_f1.2}\\
 \frac{f_1^2}{f_3^2} &= \frac{f_2f_4^2f_{12}^4}{f_6^5f_8f_{24}}-2q\frac{f_2^2f_8f_{12}f_{24}}{f_4f_6}.\label{f1.2_f3.2}\\
\frac{f_3^4}{f_1^4}
&=\frac{f_4^8 f_6^2 f_{12}^{4}}{f_2^{10} f_8^2 f_{24}^{2}}
+ 4q \frac{f_4^5 f_6^3 f_{12}}{f_2^9}
+ 4q^2 \frac{f_4^2 f_6^4 f_8^2 f_{24}^{2}}{f_2^8 f_{12}^{2}}.\label{f3f4}\\
&       \frac{1}{f_1f_3} =\frac{f_8^2 f_{12}^5}{f_2^2 f_4 f_6^4 f_{24}^2} + q \frac{f_4^5 f_{24}^2}{f_2^4 f_6^2 f_8^2 f_{12}}. \label{1byf1f3}
        \end{align}
\end{lma}

\begin{lma} In \cite{hirschhorn1993cubic}, the following 2-dissection holds
\begin{align}
    \frac{f_3^3}{f_1} &= \frac{f_4^3f_6^2}{f_2^2f_{12}}+q\frac{f_{12}^3}{f_4}.\label{f3.3_f1}
\end{align}
\end{lma}

\begin{lma} In \cite{NDBOJH}, the following 2-dissection holds
\begin{align}
    \frac{f_3}{f_1^3} = \frac{f_4^6 f_6^3}{f_2^9 f_{12}^2} + 3q \frac{f_4^2 f_6 f_{12}^2}{f_2^7}.\label{f3_f1.3}
\end{align}
\end{lma}

\begin{lma} In \cite{hirschhorn2010elementary}, the following 2-dissections holds 
\begin{equation} \label{f5_f1}
\frac{f_5}{f_1} =\frac{f_8 f_{20}^2}{f_2^2 f_{40}}
+q\frac{f_4^3  f_{10} f_{40}}{f_2^3 f_8 f_{20}}.
\end{equation}
\begin{equation}\label{f1_f5}
    \frac{f_1}{f_5}=\frac{f_2 f_8 f_{20}^3}{f_4 f_{10}^2 f_{40}}-q\frac{f_4^2f_{40}}{f_8 f_{10}^2}.
\end{equation}
\end{lma}
\begin{lma} \cite[equation 3.1]{toh2012ramanujan}, the following 3-dissection holds
    \begin{equation}\label{f2.3_f1.3}
       \frac{f_2^3}{f_1^3} =\frac{f_6}{f_3}+3q
       \frac{f_6^4 f_9^5}{f_3^8 f_{18}}
        +6q^2\frac{f_6^2 f_9^2 f_{18}^2 }{f_3^7} 
        +12q^3\frac{f_6^2 f_{18}^5}{f_3^6 f_9}.
    \end{equation}
\end{lma}

\begin{lma}In \cite{hirschhorn2014congruence}, the following 3-dissection holds  
\begin{equation} \label{f1f2}
f_1f_2=\frac{f_6f_9^4}{f_3f_{18}^2}-qf_9f_{18}-2q^2\frac{f_3f_{18}^4}{f_6f_9^2}.
\end{equation}
\end{lma}

\begin{lma}\cite[Corollary of Entry 31]{berndt2012ramanujan} We have 
    \begin{equation}\label{f2.2_f1}
        \frac{f_2^2}{f_1}=\frac{f_6f_9^2}{f_3f_{18}}+q\frac{f_{18}^2}{f_9}.
    \end{equation}
\end{lma}


\begin{lma} \cite[p. 212]{ramanujan1927} We have
     \begin{equation} \label{f_1_dissection}
        f_1=f_{25}\left( R(q^5)^{-1}-q-q^2 R(q^5)\right),
    \end{equation}
    where
    \begin{large}
        $R(q)= \frac{(q;q^5)_\infty (q^4;q^5)_\infty}{(q^2;q^5)_\infty (q^3;q^5)_\infty}.\\$ 
    \end{large}
\end{lma}



\section{Congruences for $b^{2}_{3, 2t}(n)$ and $b^{4}_{3, 4t}(n)$}\label{s3}

\begin{thm}
For any positive integer $n$ and $3 \nmid t$, we have
\begin{align}
b^{2}_{3,2t}(4n+3) \equiv 0 \pmod{4}.\label{tm1}
\end{align}
\end{thm}

\begin{proof}
Suppose $3 \nmid t$, substituting $\ell=3$, $m=2t$ and $k=2$ in \eqref{main}, we have

\begin{equation} \label{tm1_0}
\displaystyle \sum_{n=0}^\infty b^{2}_{3,2t}(n)q^n = \frac{f_3^2 f_{2t}^2}{f_1^2 f_{6t}^2}.
\end{equation}
Substituting \eqref{f3.2_f1.2} in  \eqref{tm1_0}, we get
\begin{align}
\displaystyle \sum_{n=0}^\infty b^{2}_{3,2t}(n) q^n 
&= \frac{f_4^4 f_6 f_{12}^2 f_{2t}^2}{f_2^5 f_8 f_{24} f_{6t}^2}
+2q\frac{f_4 f_6^2 f_8 f_{24} f_{2t}^2}{f_2^4 f_{12} f_{6t}^2}. \label{tm1_1}
\end{align}
By binomial theorem, for positive integer $j$, we have 
\begin{align}\label{bt2}
    f_j^2 \equiv f_{2j} \pmod{2}.
\end{align}
Extracting $q^{2n+1}$ terms from \eqref{tm1_1}, we get
\begin{align}
\displaystyle \sum_{n=0}^\infty b^{2}_{3,2t}(2n+1) q^{2n+1} 
& = 2q\frac{f_4 f_6^2 f_8 f_{24} f_{2t}^2}{f_2^4 f_{12} f_{6t}^2}, \nonumber\\
& \equiv  2q\frac{f_4^2 f_6^2 f_8 f_{24} f_{2t}^2}{f_4 f_2^4 f_{12} f_{6t}^2} \pmod{4}, \label{tm1_2}
\end{align}
which implies that 
\begin{equation} \label{tm1_4} 
\displaystyle \sum_{n=0}^\infty b^{2}_{3,2t}(2n+1) q^{n} 
\equiv  2\frac{ f_3^2 f_4 f_{12} f_{t}^2}{f_2  f_{6} f_{3t}^2} \pmod{4}.
\end{equation}
Invoking \eqref{bt2} in \eqref{tm1_4}, we have 
\begin{align}\label{l1}
\displaystyle \sum_{n=0}^\infty b^{2}_{3,2t}(2n+1) q^{n} 
& \equiv   2\frac{f_4 f_{12} f_{2t} }{f_2  f_{6t} }\pmod{4}.
\end{align}
The equation \eqref{tm1} follows from \eqref{l1}.
\end{proof}

\begin{thm} For any positive integer $n$, we have
\begin{align}
b^{2}_{3,4}(2n) \equiv b_{3,4}(n) \pmod{4}.\label{tm2}
\end{align} 
\end{thm}

\begin{proof}
Substituting $t=2$ in \eqref{tm1_0}, we have 
\begin{equation}\label{tm2_0}
\displaystyle \sum_{n=0}^\infty b^{2}_{3,4}(n)q^n = \frac{f_3^2 f_{4}^2}{f_1^2 f_{12}^2}.
\end{equation}
Substituting \eqref{f3.2_f1.2} in \eqref{tm2_0}, we have 
\begin{align}
\displaystyle \sum_{n=0}^\infty b^{2}_{3,4}(n) q^n
&= \frac{f_4^6 f_6 }{f_2^5 f_8 f_{24} }
+2q\frac{f_4^3 f_6^2 f_8 f_{24}}{f_2^4  f_{12}^3}, \label{tm2_1}
\end{align}
which implies that 
\begin{align}
\displaystyle \sum_{n=0}^\infty b^{2}_{3,4}(2n) q^{2n} 
&= \frac{f_4^6 f_6 }{f_2^5 f_8 f_{24} }, \nonumber \\
& \equiv \frac{f_4^2 f_4^4 f_6 }{f_2^4 f_2 f_8 f_{24} } \pmod{4}. \label{tm2_4}
\end{align}
Utilizing \eqref{bt2} in \eqref{tm2_4}, we find that 
\begin{align}
\displaystyle \sum_{n=0}^\infty b^{2}_{3,4}(2n) q^{2n} 
& \equiv \frac{ f_4^4 f_6 }{ f_2 f_8 f_{24} } \pmod{4}, \nonumber \\
& \equiv \frac{ f_8 f_6 }{ f_2 f_{24} } \pmod{4}. \label{tm2_2}
\end{align}
Extracting the even powers of $q$ on both sides of the above equation, we get 
\begin{align}
\displaystyle \sum_{n=0}^\infty b^{2}_{3,4}(2n) q^{n} 
& \equiv \frac{ f_3 f_4 }{ f_1 f_{12} } \pmod{4}. \label{tm2_3}
\end{align}
The equation \eqref{tm2} 
follows from \eqref{tm2_3}.
\end{proof}

\begin{thm}
For any positive integer $n$, we have
    \begin{equation} 
        b^{2}_{3,4}(4n+1)\equiv 
        \begin{cases}
            2 \pmod{4} & \text{ if $n$ is a triangular number,}\\
            0 \pmod{4} & \text{ other wise.}
        \end{cases}
    \end{equation}
\end{thm}

\begin{proof}
	Substituting $t=2$ in \eqref{l1}, we have
	\begin{align}\label{l11}
		\displaystyle \sum_{n=0}^\infty b^{2}_{3,4}(2n+1) q^{n} 
		& \equiv   2\frac{f^2_4 }{f_2}\pmod{4}.
	\end{align}
Extracting the even powers of $q$  on both sides of the above equation, we get
\begin{align}\label{l12}
	\displaystyle \sum_{n=0}^\infty b^{2}_{3,4}(4n+1) q^{n} 
	& \equiv   2\frac{f^2_2}{f_1} \equiv   2\psi(q) \pmod{4},
\end{align}
which implies that 
\begin{align}\label{l121}
	\displaystyle \sum_{n=0}^\infty b^{2}_{3,4}(4n+1) q^{n} 
	& \equiv  2\sum_{n=0}^{\infty}q^\frac{n(n+1)}{2}  \pmod{4}.
\end{align}
The result follows.
\end{proof}
\begin{cor}
Let $p \geq 5$ be prime and let $1 \leq r \leq p - 1$ with $8r + 1$ a quadratic nonresidue modulo $p$. Then for all $m \geq 0$,    
\end{cor}
\begin{equation}
    b_{3,4}^{2} (pm + r) \equiv 0 \pmod{4}.
\end{equation}

\begin{proof}
    We have
\begin{equation}
    \sum_{n=0}^{\infty} b_{3,4}^{2} (n) q^{8n+1} \equiv \sum_{k=0}^{\infty} q^{(2k+1)^2} \pmod{4}.
\end{equation}
Note that here, $n = pm + r$, so $8n+1 = 8pm+8r + 1 \equiv 8r +1 \pmod{p}$ is not a square modulo $p$. Thus, $8n + 1$ is not a square and $b_{3,4}^{2}(n) \equiv 0 \pmod{8}$. 
\end{proof} 
 \begin{thm} For any positive integer $n$, we have
     \begin{align}
         &b_{3,4}^2(6n+4) \equiv 0 \pmod{2}, \label{tmx1} \\
         & \displaystyle b^{2}_{3,4}(6n+2) q^{n} 
      \equiv b^3_3(n)\pmod{2},\label{tmx2}\\
    & b_{3,4}^2(12n+6) \equiv 0 \pmod{2}, \label{tmx3}\\
   & b_{3,4}^2(12n) \equiv p(n) \pmod{2}. \label{tmx4}
     \end{align}
 \end{thm}
 \begin{proof} Substituting $t=2k$ in \eqref{tm1_1} and extracting only even power of $q$ on both sides of the resulting equation, we get
     \begin{align}
 \displaystyle \sum_{n=0}^\infty b^{2}_{3,4k}(2n) q^{2n} 
&= \frac{f_4^4 f_6 f_{12}^2 f_{4k}^2}{f_2^5 f_8 f_{24} f_{12k}^2}.
\end{align}
Thanks to \eqref{bt2}, we have
\begin{align}
 \displaystyle \sum_{n=0}^\infty b^{2}_{3,4k}(2n) q^{n} 
& \equiv \frac{f_3 f_{2k}^2}{f_1 f_{6k}^2}\pmod{2}. \label{tmx_1}
\end{align}
Now, Substituting $k=1$ in the above equation \eqref{tmx_1} and utilizing \eqref{f2.2_f1}, we have
\begin{align}
     \displaystyle \sum_{n=0}^\infty b^{2}_{3,4}(2n) q^{n} & \equiv \frac{f_9^2}{f_6f_{18}}+q\frac{f_3 f_{18}^2}{f_6^2 f_9}\pmod{2},\\
     & \equiv \frac{1}{f_6}+q\frac{f_{9}^3}{f_3^3}\pmod{2}\label{tmx_3},
\end{align}
which implies that
\begin{align}
     \displaystyle \sum_{n=0}^\infty b^{2}_{3,4}(6n) q^{n} 
     & \equiv \frac{1}{f_2}\pmod{2} \label{tmx_4}.
\end{align}
\noindent Extracting the terms containing $q^{3n+1}$ from \eqref{tmx_3}, we obtain     
     \begin{align}
      \displaystyle \sum_{n=0}^\infty b^{2}_{3,4}(6n+2) q^{n} 
     & \equiv \frac{f_3^3}{f_1^3}\pmod{2} \label{tmx_6}.
\end{align}
 Congruence \eqref{tmx1} follows from \eqref{tmx_3}. Using \eqref{bkln} and \eqref{tmx_6}, we arrive at \eqref{tmx2}. The equations \eqref{tmx3} and \eqref{tmx4} follows from \eqref{tmx_4} . This completes the proof.
 \end{proof}
  \begin{thm}For any positive integer $n$, we have
\begin{align}
 b^{2}_{3,8}(2n+1) \equiv 2b_{2,3}(n) \pmod{3}.\label{tm6}
\end{align}
\end{thm}

\begin{proof}
Setting $t=4$ in \eqref{tm1_0}, we have
\begin{equation}
\displaystyle \sum_{n=0}^\infty  b^{2}_{3,8}(n)q^n=\frac{f_3^2 f_8^2}{f_1^2 f_{24}^2}.
\end{equation}
Substituting \eqref{f3.2_f1.2} in above equation, we have
\begin{align}
\displaystyle \sum_{n=0}^\infty  b^{2}_{3,8}(n)q^n 
&=\frac{f_4^4 f_6 f_8 f_{12}^2}{f_2^5 f_{24}^3}
+2q\frac{f_4 f_6^2 f_8^3 }{f_2^4 f_{12} f_{24}}.\label{tm6_1}
\end{align}
Extracting $q^{2n+1}$ terms from \eqref{tm6_1}, we get 
\begin{align}
\displaystyle \sum_{n=0}^\infty  b^{2}_{3,8}(2n+1)q^{2n+1} &=2q\frac{f_4 f_6^2 f_8^3 }{f_2^4 f_{12} f_{24}}. \label{tm6_2} 
\end{align}
By binomial theorem, we have
\begin{equation}\label{bt3}
f_t^3 \equiv f_{3t} \pmod{3}.  
\end{equation}
Invoking \eqref{bt3} in \eqref{tm6_2}, we have    
\begin{align}
\displaystyle \sum_{n=0}^\infty  b^{2}_{3,8}(2n+1)q^{2n+1} 
&\equiv 2q\frac{f_4 f_6^2 f_8^3 }{f_2^4 f_{12} f_{24}} \pmod 3, \nonumber\\
&\equiv 2q\frac{ f_4 f_6 f_8^3 }{f_2 f_{12} f_{24}} \pmod 3, \nonumber \\
&\equiv 2q\frac{ f_4 f_6  }{f_2 f_{12} } \pmod 3. \label{tm6_3}
\end{align}
Congruence \eqref{tm6} follows from \eqref{tm6_3}. This completes the proof\end{proof}

\begin{thm}
	For any positive integer $n$ and $3\nmid t$, we have
	\begin{align}
		&b^{4}_{3,4t}(16n+15) \equiv 0 \pmod{8}.\label{2ntm1}\\
		&b^{4}_{3,4t}(16n+11) \equiv 0 \pmod{8}.\label{2ntm2}\\
        &b^{4}_{3,4}(48n+8i+7) \equiv 0 \pmod{8},\,\,\text{for}\,\, i={2,3,4,5}.\label{2ntm6}\\
        &b^{4}_{3,4}(2n+1) \equiv \begin{cases} 
    4 \pmod{3}, & \text{if } n = 0, \\
    0 \pmod{3}, & \text{otherwise.}
    \end{cases}\label{2ntm4}	
    \end{align}
\end{thm}
\begin{proof}
	Suppose $3 \nmid t$, substituting $\ell=3$, $m=4t$ and $k=4$ in \eqref{main}, we have
	\begin{equation} \label{2tm1_0}
		\displaystyle \sum_{n=0}^\infty b^{4}_{3,4t}(n)q^n = \frac{f_3^4 f_{4t}^4}{f_1^4 f_{12t}^4}.
	\end{equation}
	Substituting \eqref{f3f4} into \eqref{2tm1_0}, we get
	\begin{align}
		\displaystyle \sum_{n=0}^\infty b^{4}_{3,4t}(n) q^n 
		&= 		
		\frac{f_4^8 f_6^2 f_{12}^{4}f_{4t}^4}{f_2^{10} f_8^2 f_{24}^{2}f_{12t}^4}
		+ 4q \frac{f_4^5 f_6^3 f_{12}f_{4t}^4}{f_2^9f_{12t}^4}
		+ 4q^2 \frac{f_4^2 f_6^4 f_8^2 f_{24}^{2}f_{4t}^4}{f_2^8 f_{12}^{2}f_{12t}^4}. \label{2tm1_1}
	\end{align}
	Extracting the terms containing $q^{2n+1}$ from  \eqref{2tm1_1}, we get
	\begin{align}
	\displaystyle \sum_{n=0}^\infty b^{4}_{3,4t}(2n+1) q^{n} & \equiv  4\frac{f_2f_{6}f_{2t}^4}{f_{6t}^4}\frac{f_3^3}{f_1} \pmod{8}.\label{2tm1_2}
	\end{align}
	Invoking \eqref{f3.3_f1} in \eqref{2tm1_2} and then extracting odd powers of $q$ in the resulting equation, we get
	\begin{equation} \label{2tm1_4} 
		\displaystyle \sum_{n=0}^\infty b^{4}_{3,4t}(4n+3) q^n \equiv  4 \frac{f_{2t}^2f_6^3}{f^2_{6t}}\frac{f_3}{f_1} \pmod{8}.
	\end{equation}
	Substituting \eqref{f3f1} into \eqref{2tm1_4}, we have 
	\begin{align}\label{2l1}
		\displaystyle \sum_{n=0}^\infty b^{4}_{3,4t}(4n+3) q^n \equiv 4\frac{f_{2t}^2f_6^3}{f^2_{6t}}\left[
		\frac{f_4 f_6 f_{16}f_{24}^{2}}{f_2^{2} f_8 f_{12}f_{48}}
		+ q \frac{f_6 f_8^2 f_{48}}{f_2^2f_{16}f_{24}}\right]\pmod{8}.
	\end{align}
    Extracting the terms containing odd powers and even powers of $q$ respectively from the above equation, we get  
    \begin{align}\label{2l12}
		\displaystyle \sum_{n=0}^\infty b^{4}_{3,4t}(8n+3) q^n \equiv 4\frac{f_{t}^2f_2f_3^4f_8f_{12}^2}{f_{3t}^2f_1^2f_4f_{6}f_{24}} \equiv 4\frac{f_{2t}f_4^2f_6^2f_{12}^2}{f_{6t}f_4f_6f_{24}}\pmod{8}
	\end{align}
     and 
	\begin{align}\label{2l11}
		\displaystyle \sum_{n=0}^\infty b^{4}_{3,4t}(8n+7) q^n \equiv 4\frac{f_{t}^2f_4^2f_3^4f_{24}}{f_{3t}^2f_1^2f_8f_{12}} \equiv 4\frac{f_{2t}f_4^2f_6^2f_{24}}{f_{6t}f_2f_8f_{12}}\pmod{8}.
	\end{align}
	Congruence \eqref{2ntm1} and \eqref{2ntm2} respectively follows from the equations \eqref{2l11} and \eqref{2l12} respectively.
    \noindent From \eqref{2l11}, we have 
	\begin{align}\label{2l13}
		\displaystyle \sum_{n=0}^\infty b^{4}_{3,4}(8n+7) q^n \equiv 4f_6f_{12}\pmod{8}.
	\end{align}
Congruence \eqref{2ntm6} is true from the above equation.

\noindent Substituting $t=1$ in \eqref{2tm1_1}, we get
\begin{align}
		\displaystyle \sum_{n=0}^\infty b^{4}_{3,4}(2n+1) q^n \equiv
		4 \frac{f_2^9 f_3^3}{f_1^9f_{6}^3} \pmod 3\label{2tm1_11}
	\end{align}
Thanks to the binomial theorem, we arrive at \eqref{2ntm4}.
\end{proof}
\begin{thm}
	If $n$ cannot be represented as the sum of a pentagonal number and two times a pentagonal number, then 
	\begin{align}
    	&b^{4}_{3,4}(16n+3) \equiv 0 \pmod{8}.\label{2ntm51}\\
		&b^{4}_{3,4}(48n+7) \equiv 0 \pmod{8}.\label{2ntm71}
	\end{align}
\end{thm}
\begin{proof}
	Substituting $t=1$ into \eqref{2l12} and equating the even powers of $q$ on both sides in the resulting equation, we get 
    \begin{align}
		b^{4}_{3,4}(16n+3) &\equiv 4f_1f_2 \pmod{8},\label{2ntm5}\\
        &\equiv \sum_{k,l=-\infty}^{\infty} (-1)^{k+l}q^{\frac{k(3k-1)}{2}+{l(3l-1)}} \pmod 8\label{2ntm512}
        \end{align}
    which yields \eqref{2ntm51}. Now using \eqref{2l13}, we get
    \begin{align}
		\displaystyle \sum_{n=0}^\infty b^{4}_{3,4}(48n+7) q^n &\equiv 4f_1f_{2}\pmod{8}.\label{2ntm7}\\
        &\equiv \sum_{k,l=-\infty}^{\infty} (-1)^{k+l}q^{\frac{k(3k-1)}{2}+{l(3l-1)}}. \pmod 8\label{2ntm712}
	\end{align}
    This completes the proof of \eqref{2ntm71}.
\end{proof}
\begin{cor}
Let $p \geq 5$ be a prime and $\left( \frac{-2}{p} \right) = -1$. Then for all $k,m \geq 0$ with $p \nmid m$,
\begin{align}
    b_{3,4}^{4} \left(16\times p^{2k+1}m+ 2 \times p^{2k+2} + 1 \right) &\equiv 0 \pmod{8},\label{cks1}\\
    b_{3,4}^{4} \left(48\times p^{2k+1}m+ 6\times p^{2k+2} + 1 \right) &\equiv 0 \pmod{8}.\label{cks3}
\end{align}
\end{cor}
\begin{proof}
    From \eqref{2ntm512}, we have
\begin{equation}\label{cks2}
    \sum_{n=0}^{\infty} b_{3,4}^{4}(16n+3)q^{24n+3} \equiv \sum_{k,l=-\infty}^{\infty} (-1)^{k+l} q^{(6k-1)^2+2(6l-1)^2} \pmod{8}.
\end{equation}
Thus, if $24n + 3$ is not of the form $(6k - 1)^2 + 2(6l - 1)^2$, then $b_{3,4}^{4}(16n+3) \equiv 0 \pmod{8}$. If $N$ is of the form $x^2 + 2y^2$, then $\nu_p(N)$ is even since $   \left( \frac{-2}{p} \right) = -1.$ However, here $n = p^{2k+1}m+\dfrac{p^{2k+2}-1}{8}$ and $\nu_p(24n+3)$ is odd. So $24n+3$ is not of the form $x^2 + 2y^2$ and $b_{3,4}^{4}(16n+3) \equiv 0 \pmod{8}$. This completes the proof of \eqref{cks1}. Since the proof of equation \eqref{cks3} follows a similar argument as that of \eqref{cks1}, we omit the details.
\end{proof} 

\begin{thm}We have for all $n\geq 0,$
\begin{align}
    &\sum_{n=0}^\infty b^{4}_{3,2}(2n) q^{n}= \dfrac{f_2^9f_3}{f_1^{7}f_{6}^3}+3q\frac{f_2f_{6}^5}{f_1^3f_3^3}\label{b241},\\
 &\sum_{n=0}^\infty b^{4}_{3,2}(2n+1) q^n= 4\dfrac{f_2^5f_{6}}{f_1^5f_3},\label{2ntm311}\\
&b^{4}_{3,2}(16n+9) \equiv 0\pmod 8,\label{cks17}\\
&b^{4}_{3,2}(16n+13) \equiv 0\pmod 8.\label{cks15}\\
&\sum_{n=0}^\infty b^{4}_{3,2}(8n+5) q^n\equiv 4{f_1f_{6}}\pmod 8,\label{cks221}
\end{align}
   \end{thm}
\begin{proof}
    \noindent Multiplying \eqref{f3.3_f1} and \eqref{f3_f1.3}, we have
\begin{equation}\label{cks51}
    \frac{f_3^4}{f_1^4} = \frac{f_4^9 f_6^5}{f_2^{11} f_{12}^3} + 3q^2 \frac{f_4 f_6 f_{12}^5}{f_2^7} + 4q \frac{f_4^5 f_6^3 f_{12}}{f_2^9}.
\end{equation}
Applying the identity in \eqref{main} with $\ell=3$, $m=2$ and $k=4$, we have 
\begin{equation} \label{2tm1_10}
		\displaystyle \sum_{n=0}^\infty b^{4}_{3,2}(n)q^n = \frac{f_3^4 f_{2}^4}{f_1^4 f_{6}^4}.
	\end{equation}
Invoking \eqref{cks51} and extracting the terms containing $q^{2n}$ and $q^{2n+1}$ respectively from the resulting equation, we arrive at \eqref{b241} and \eqref{2ntm311} respectively. Multiplying \eqref{1_f1.4} and \eqref{1byf1f3}, we find that
\begin{equation}
    \dfrac{1}{f_1^5f_3} = \frac{f_4^{13} f_{12}^5}{f_2^{16} f_8^2 f_6^4 f_{24}^2}
    + q \frac{f_4^{19} f_{24}^2}{f_2^{18} f_8^6 f_6^2 f_{12}}
    + 4q \frac{f_4 f_8^6 f_{12}^5}{f_2^{12} f_6^4 f_{24}^2}
    + 4q^2 \frac{f_4^7 f_8^2 f_{24}^2}{f_2^{14} f_6^2 f_{12}}.\label{cks53}
\end{equation}
Substituting \eqref{cks53} in \eqref{2ntm311} and then extracting the terms containing $q^{2n}$ and $q^{2n+1}$ respectively from the resulting equation, we get
\begin{align}
            &\sum_{n=0}^\infty b^{4}_{3,2}(4n+1) q^{n}= 4\frac{f_2^{13} f_6^5}{f_1^{11} f_4^2 f_3^3 f_{12}^2} + 16q \frac{f_2^7 f_4^2 f_{12}^2}{f_1^9 f_3 f_{6}},\label{cks6}\\
         &\sum_{n=0}^\infty b^{4}_{3,2}(4n+3) q^{n} = 
    4\frac{f_2^{19} f_{12}^2}{f_1^{13} f_4^6 f_3 f_{6}}
    + 16 \frac{f_2f_4^6 f_{6}^5}{f_1^7 f_3^3 f_{12}^2}.\label{cks7}
\end{align}
 Using \eqref{bt2} in \eqref{cks6}, we have 
\begin{align}\label{cks11}
    \sum_{n=0}^\infty b^{4}_{3,2}(4n+1) q^n \equiv 4\frac{f_2^4}{f_1f_3}\pmod 8.
\end{align}
Invoking \eqref{1byf1f3} in \eqref{cks11}, we have
\begin{align}\label{cks12}
    \sum_{n=0}^\infty b^{4}_{3,2}(4n+1) q^n \equiv 4\left[\frac{f_2^2f_8^2 f_{12}^5}{f_4 f_6^4 f_{24}^2} + q \frac{f_4^5 f_{24}^2}{f_6^2 f_8^2 f_{12}}\right]\pmod 8.
\end{align}
Using \eqref{bt2} in \eqref{cks12}, we get
\begin{align}\label{cks13}
    \sum_{n=0}^\infty b^{4}_{3,2}(8n+1) q^n \equiv 4\frac{f_2^3f_1^2}{f_3^2}\pmod 8
\end{align}
and 
\begin{align}\label{cks14}
    \sum_{n=0}^\infty b^{4}_{3,2}(8n+5) q^n \equiv 4\frac{f_6^3f_1^2}{f_3^2}\pmod 8
\end{align}
Substituting \eqref{f1.2_f3.2} in \eqref{cks13}, and extracting odd powers of $q$ on both sides of the equation, we arrive at \eqref{cks17}.
Substituting \eqref{f1.2_f3.2} in \eqref{cks14}, we get 
\begin{align}\label{cks16}
    \sum_{n=0}^\infty b^{4}_{3,2}(8n+5) q^n \equiv 4\frac{f_2f_{12}^2}{f_6^2}\pmod 8.
\end{align}
Congruence \eqref{cks15} follows by equating odd powers of $q$ from above equation, \eqref{cks221} is obtained from \eqref{cks16}. This completes the proof.
\end{proof}
\begin{thm}
	If $n$ cannot be represented as the sum of a pentagonal number and six times a pentagonal number, then 
	\begin{align}\label{cks21}
    	&b^{4}_{3,2}(16n+5) \equiv 0 \pmod{8}.
	\end{align}
\end{thm}
\begin{proof}
Observe \eqref{cks16}, we have  
    \begin{align}
		b^{4}_{3,4}(16n+5)q^n &\equiv 4f_1f_6 \pmod{8},\label{2ntm5}\\
        &\equiv \sum_{k,l=-\infty}^{\infty} (-1)^{k+l}q^{\frac{k(3k-1)}{2}+{3l(3l-1)}} \pmod 8.\label{cks22}
        \end{align}
    This completes the proof of \eqref{cks21}.    
\end{proof}
\begin{cor}
Let $p \geq 5$ be a prime and $\left( \frac{-6}{p} \right) = -1$. Then for all $k,m \geq 0$ with $p \nmid m$,
\begin{align}
    b_{3,2}^{4} \left(16\times p^{2k+1}m+ \frac{2\times p^{2k+2}+1}{3} \right) &\equiv 0 \pmod{8}.\label{cks22}
\end{align}
\end{cor}
\begin{proof}
    From \eqref{cks16}, we have
\begin{equation}\label{cks23}
    \sum_{n=0}^{\infty} b_{3,2}^{4}(16n+5)q^{24n+7} \equiv \sum_{k,l=-\infty}^{\infty} (-1)^{k+l} q^{(6k-1)^2+6(6l-1)^2} \pmod{8}.
\end{equation}
Thus, if $24n + 7$ is not of the form $(6k - 1)^2 + 6(6l - 1)^2$, then $b_{3,2}^{4}(16n+5) \equiv 0 \pmod{8}$. If $N$ is of the form $x^2 + 6y^2$, then $\nu_p(N)$ is even since $   \left( \frac{-6}{p} \right) = -1.$ However, here $n = p^{2k+1}m+\dfrac{7p^{2k+2}-7}{24}$ and $\nu_p(24n+7)$ is odd. So $24n+7$ is not of the form $x^2 + 6y^2$ and $b_{3,2}^{4}(16n+5) \equiv 0 \pmod{8}$. This completes the proof of \eqref{cks22}. 
\end{proof} 
\begin{thm}
		For positive integers $j,n \geq 0$, we have
        \begin{equation}\label{thyn0}
            b_{2,5}^4(2n+1)\equiv 0\pmod{4},
        \end{equation}
		\begin{equation} \label{thyn1}
			\displaystyle \sum_{n=0}^\infty b_{2,5}^4\left( 4\times 5^{2j} n + \frac{13\times 5^{2j} -1}{6}\right) q^n \equiv 2 f_1 f_2 f_{10} \pmod{4},
		\end{equation}
		and for $y \in \{41,89\}$, we have 
		\begin{equation} \label{thyn2}
			 b_{2,5}^4\left( 4\times 5^{2j+2} n + \frac{y\times 5^{2j+1} -1}{6}\right) \equiv 0 \pmod{4}.
		\end{equation}
\end{thm}
\begin{proof}
		Setting $\ell=2$, $m=5$ and $k=4$ in \eqref{main}, we obtain
		\begin{equation} \label{thyn1_0}
			\displaystyle \sum_{n=0}^\infty b_{2,5}^4(n)q^n = \frac{f_2^4 f_5^4}{f_1^4 f_{10}^4}.
		\end{equation}
		Using binomial theorems $f_1^4 \equiv f_2^2 \pmod{4}$ and $ f_5^4 \equiv  f_{10}^2 \pmod{4}$  in \eqref{thyn1_0}, we have
		\begin{equation} \label{thyn1_2}
			\displaystyle \sum_{n=0}^\infty b_{2,5}^4(n)q^{n}  \equiv \frac{ f_2^2 }{f_{10}^2} \pmod{4}. 
		\end{equation}
		extracting odd and even powers of $ q $ in equation \eqref{thyn1_2} we get
		\begin{align}
		 b_{2,5}^4(2n+1)q^{n}  \equiv 0 \pmod{4}, \\
			\displaystyle \sum_{n=0}^\infty b_{2,5}^4(2n)q^{2n}  \equiv \frac{ f_1^2 }{f_{5}^2} \pmod{4}.  \label{thyn1_3}
		\end{align}
		Squaring \eqref{f1_f5}, we have 
		\begin{equation}
			\frac{ f_1^2 }{f_5^2} =\frac{ f_2^2f_8^2 f_{20}^6 }{f_4^2 f_{10}^4f_{40}^2} +q^2\frac{ f_4^4 f_{40}^2 }{ f_8^2 f_{10}^2} -2q\frac{ f_2 f_4 f_{20}^3}{ f_{10}^4 }. \label{srr}
		\end{equation}
		Invoking \eqref{srr} in \eqref{thyn1_3} and extracting only odd powers of $ q $ we get 
		\begin{equation}
			\displaystyle \sum_{n=0}^\infty b_{2,5}^4(4n+2)q^{n}  \equiv 2f_1 f_2  f_{10} \pmod{4}.  \label{thyn1_4}
		\end{equation}
		Above equation \eqref{thyn1_4} proves the case of $ j=0 $ in equation \eqref{thyn1}.
		Now if the congruence \eqref{thyn1} holds for some integer $j \geq 0$ then 
		invoking \eqref{f_1_dissection} in \eqref{thyn1}, we get 
		\begin{align}
			\displaystyle \sum_{n=0}^\infty b_{2,5}^5\left( 4\times 5^{2j} n + \frac{13\times 5^{2j} -1}{6}\right) q^n &\equiv 2 f_{10} f_{25} f_{50}\left( R(q^5)^{-1}-q-q^2 R(q^5)\right) \nonumber \\ 
			&\quad \times \left( R(q^{10})^{-1}-q^2-q^4 R(q^{10})\right) \pmod{4}, \label{thyn1_5} 
		\end{align}
		Extracting only $q^{5n+3}$ terms on both sides of \eqref{thyn1_5}, we have
		\begin{align}
			\displaystyle \sum_{n=0}^\infty b_{2,5}^5\left( 4\times 5^{2j} (5n+3) + \frac{13\times 5^{2j} -1}{6}\right) q^{5n+3} &\equiv 2 q^3 f_{10} f_{25} f_{50} \pmod{4},\nonumber \\
			\displaystyle \sum_{n=0}^\infty b_{2,5}^4\left( 4\times 5^{2j+1}n + \frac{17\times 5^{2j+1} -1}{6}\right) q^n &\equiv 2 f_2 f_{5} f_{10} \pmod{4}.
		\end{align}
		Again, substituting \eqref{f_1_dissection} in the above equation, we have
		\begin{align} 
			\displaystyle \sum_{n=0}^\infty b_{2,5}^4\left( 4\times 5^{2j+1}n + \frac{17\times 5^{2j+1} -1}{6}\right) q^n & \equiv 2 f_{5} f_{10} f_{50}  \left(R(q^{10})^{-1}-q^2-q^4 R(q^{10})\right) \pmod{4}. \label{thyn1_6}
		\end{align}
		Extracting only $q^{5n+2}$ terms from \eqref{thyn1_6}, we have
		\begin{align}
			\displaystyle \sum_{n=0}^\infty b_{2,5}^4\left( 4\times 5^{2j+1}(5n+2) + \frac{17\times 5^{2j+1} -1}{6}\right) q^{5n+1} & \equiv -2q^2 f_{5} f_{10} f_{50} \pmod{4}. \nonumber \\
			\displaystyle \sum_{n=0}^\infty b_{2,5}^4\left( 4\times 5^{2j+2}n + \frac{13\times 5^{2j+2} -1}{6}\right) q^n & \equiv 2 f_{1} f_{2} f_{10} \pmod{4}.
		\end{align}
		The above equation proves the congruence \eqref{thyn1} for $j+1$, hence by induction \eqref{thyn1} is true. The congruence \eqref{thyn2} follows from \eqref{thyn1_6}. This completes the proof.
\end{proof}
    
\section{Congruences for $b^{2}_{2,5}(n)$}\label{s4}
\begin{thm}
    For $j \geq 0$, we have
    \begin{equation} \label{thm1}
         \displaystyle \sum_{n=0}^\infty b_{2,5}^2\left( 2\times 5^{2j} n + \frac{2\times 5^{2j} +1}{3}\right) q^n \equiv 2 f_1 f_2 f_5 \pmod{4},
    \end{equation}
    and for $t \in \{22,28\}$, we have 
    \begin{equation} \label{thm1__}
        b_{2,5}^2\left( 2\times 5^{2j+2} n + \frac{t\times 5^{2j+1} +1}{3}\right) \equiv 0 \pmod{4}.
    \end{equation}
\end{thm}
\begin{proof}
	Setting $\ell=2$, $m=5$ and $k=2$ in \eqref{main}, we obtain
    \begin{equation} \label{thm1_0}
        \displaystyle \sum_{n=0}^\infty b_{2,5}^2(n)q^n = \frac{f_2^2 f_5^2}{f_1^2 f_{10}^2}.
    \end{equation}
    Invoking \eqref{f5_f1} in \eqref{thm1_0} and extracting odd powers of $q$ on both sides of \eqref{thm1_0}, we get
     \begin{align}
         \displaystyle \sum_{n=0}^\infty b_{2,5}^2(2n+1)q^{2n+1} &= 2q \frac{ f_4^3 f_{20}}{f_2^3 f_{10}}. \nonumber \\
          \displaystyle \sum_{n=0}^\infty b_{2,5}^2(2n+1)q^{n} &=
         2\frac{ f_2^3 f_{10}}{f_1^3 f_{5}}.\\
         & \equiv 2\frac{f_1  f_2^3 f_5 f_{10}}{f_1^4 f_{5}^2} \pmod{4}. \label{thm1_1}
     \end{align} 
    Using binomial theorems $f_1^4 \equiv f_2^2 \pmod{4}$ and $2 f_5^2 \equiv 2 f_{10} \pmod{4}$  in \eqref{thm1_1}, we have
    \begin{equation} \label{thm1_2}
        \displaystyle \sum_{n=0}^\infty b_{2,5}^2(2n+1)q^{n}  \equiv 2 f_1  f_2 f_5\pmod{4}. 
    \end{equation}
    The above equation \eqref{thm1_2} proves the case for $j=0.$
    
    Now if the congruence \eqref{thm1} holds for some integer $j \geq 0$ then 
    invoking \eqref{f_1_dissection} in \eqref{thm1}, we get 
    \begin{align}
          \displaystyle \sum_{n=0}^\infty b_{2,5}^2\left( 2\times 5^{2j} n + \frac{2\times 5^{2j} +1}{3}\right) q^n &\equiv 2 f_5 f_{25} f_{50}\left( R(q^5)^{-1}-q-q^2 R(q^5)\right) \nonumber \\ 
         &\quad \times \left( R(q^{10})^{-1}-q^2-q^4 R(q^{10})\right)  \pmod{4}, \label{thm1_3} 
    \end{align}
Extracting only $q^{5n+3}$ terms on both sides of \eqref{thm1_3}, we have
\begin{align}
    \displaystyle \sum_{n=0}^\infty b_{2,5}^2\left( 2\times 5^{2j} (5n+3) + \frac{2\times 5^{2j} +1}{3}\right) q^{5n+3} &\equiv 2 q^3 f_5 f_{25} f_{50} \pmod{4},\nonumber \\
     \displaystyle \sum_{n=0}^\infty b_{2,5}^2\left( 2\times 5^{2j+1}n + \frac{4\times 5^{2j+1} +1}{3}\right) q^n &\equiv 2 f_1 f_{5} f_{10} \pmod{4}.
\end{align}
Again, substituting \eqref{f_1_dissection} in the above equation, we have
\begin{align} 
    \displaystyle \sum_{n=0}^\infty b_{2,5}^2\left( 2\times 5^{2j+1}n + \frac{4\times 5^{2j+1} +1}{3}\right) q^n & \equiv 2 f_{5} f_{10} f_{25} \left(R(q^5)^{-1}-q-q^2 R(q^5)\right) \pmod{4}. \label{thm1_4}
\end{align}
Extracting only $q^{5n+1}$ terms from \eqref{thm1_4}, we have
\begin{align}
    \displaystyle \sum_{n=0}^\infty b_{2,5}^2\left( 2\times 5^{2j+1}(5n+1) + \frac{4\times 5^{2j+1} +1}{3}\right) q^{5n+1} & \equiv -2q f_{5} f_{10} f_{25} \pmod{4}. \nonumber \\
     \displaystyle \sum_{n=0}^\infty b_{2,5}^2\left( 2\times 5^{2j+2}n + \frac{2\times 5^{2j+2} +1}{3}\right) q^n & \equiv 2 f_{1} f_{2} f_{5} \pmod{4}.
\end{align}
The above equation proves the congruence \eqref{thm1} for $j+1$, hence by induction \eqref{thm1} is true. Congruence \eqref{thm1__} follows from \eqref{thm1_4}. This completes the proof.
\end{proof}

\section{Congruences for $b^{2}_{5, 4t}(n)$}\label{s5}
\begin{thm} For any nonnegative integer $n$ and $5 \nmid t$ and, we have 
\begin{align} 
 b^{2}_{5, 4t}(4n+3) &\equiv 0 \pmod{10} .\label{tm5}
\end{align}
\end{thm}

\begin{proof}
Setting $\ell=4t$, $m=5$ and $k=2$ in \eqref{main}, we have
\begin{equation}
\displaystyle \sum_{n=0}^\infty  b^{2}_{4t,5}(n)q^n = \frac{f_{4t}^2 f_{5}^2}{f_1^2 f_{20t}^2}.\label{4t5}
\end{equation}
Substituting \eqref{f5_f1} in the above equation, we have 
\begin{align}
\displaystyle \sum_{n=0}^\infty  b^{2}_{5, 4t}(n)q^n 
&= \frac{f_{4t}^2 f_8^2 f_{20}^4}{f_{20t}^2 f_2^4 f_{40}^2}
+q^2\frac{f_{4t}^2 f_4^6  f_{10}^2 f_{40}^2}{f_{20t}^2 f_2^6 f_8^2 f_{20}^2} +2q\frac{f_{4t}^2 f_4^3 f_{10} f_{20}}{f_{20t}^2 f_2^5 }.\label{tm5_1}
\end{align}
Extracting only $q^{2n+1}$ terms from \eqref{tm5_1}, we get
\begin{align}
\displaystyle \sum_{n=0}^\infty  b^{2}_{5, 4t}(2n+1)q^n &=
2\frac{f_{2t}^2  f_2^3 f_5 f_{10}}{f_{10t}^2 f_1^5}. \label{tm5_2}
\end{align}
Applying binomial theorem $ 2f_5 \equiv 2 f_1^5 \pmod{10} $ in \eqref{tm5_2}, we have
\begin{align}
\displaystyle \sum_{n=0}^\infty b^{2}_{5, 4t}(2n+1)q^n 
&\equiv 2\frac{f_{2t}^2  f_2^3 f_5 f_{10}}{f_{10t}^2 f_1^5} \pmod{10}, \nonumber\\
&\equiv 2\frac{f_{2t}^2  f_2^3  f_{10}}{f_{10t}^2 } \pmod{10}.\label{l2}
\end{align}
Congruence \eqref{tm5} follows from \eqref{l2}.
\end{proof}

\begin{thm}
For any integer  $n\ge0$, we have
\begin{equation} \label{b2_45_2n1}
b^{2}_{4,5}(8n+7) \equiv 0 \pmod{40}. 
\end{equation}
\end{thm}

\begin{proof}
Substituting $t=1$ in \eqref{4t5} and invoking \eqref{f5_f1}, we have
\begin{align}
\displaystyle \sum_{n=0}^\infty b^{2}_{4,5}(n)q^n 
&= \frac{f_4^2 f_8^2 f_{20}^4}{f_{20}^2 f_2^4 f_{40}^2}
+q^2\frac{f_4^8  f_{10}^2 f_{40}^2}{f_2^6 f_8^2 f_{20}^4} +2q\frac{f_4^5 f_{10} }{f_2^5 f_{20}}. \label{tm3_1}
\end{align}
Extracting odd powers of $q$ on both sides of \eqref{tm3_1}, we get
\begin{align}
\displaystyle \sum_{n=0}^\infty b^{2}_{4,5}(2n+1)q^{n} &= 2\frac{f_2^5 f_{5} }{f_1^5 f_{10}}=\dfrac{f_2^5}{f_{10}}\dfrac{1}{f_1^4}\dfrac{f_5}{f_1} \label{tm3_2}.
\end{align}
Using \eqref{1_f1.4} and \eqref{f5_f1} in \eqref{tm3_2} and extracting only odd terms, we have
 \begin{align}
\displaystyle \sum_{n=0}^\infty b^{2}_{4,5}(4n+3)q^{n} 
&= 2\frac{f_2^{17} f_{20} }{f_1^{12} f_{4}^4}+8\frac{f_4^5 f_2^2 f_{10}^2}{f_1^{12} f_{20}}, \nonumber\\
& \equiv 2\frac{f_2^{17} f_{20} }{f_1^{12} f_{4}^4} \pmod{8}. \label{tm4_0}
\end{align}
By binomial theorem, we have
\begin{align}
\displaystyle \sum_{n=0}^\infty b^{2}_{4,5}(4n+3)q^{n} 
    \equiv 2 f_2^{3} f_{20} \pmod{8}.\label{tm4_1}
    \end{align}
which implies that
\begin{equation}\label{tm3_42}
b^{2}_{4,5}(8n+7)\equiv 0 \pmod{8}.
\end{equation}
Thanks to the binomial theorem, we obtain
\begin{align}\label{tm3_41}
\displaystyle \sum_{n=0}^\infty b^{2}_{4,5}(2n+1)q^{n}\equiv 2\frac{f_{10} f_{5} }{f_5 f_{10}} \equiv 2 \pmod{10}
\end{align}
Now as  $b^{2}_{3,4}(1)=2$, congruence \eqref{b2_45_2n1} follows from \eqref{tm3_42} and \eqref{tm3_41}.
\end{proof}

\begin{thm}
    For $j \geq 0$, we have
    \begin{equation} \label{thm2}
         \displaystyle \sum_{n=0}^\infty b_{4,5}^2\left( 8\times 5^{2j} n + \frac{13\times 5^{2j} -4}{3}\right) q^n \equiv 2 f_1^3 f_{10} \pmod{8}.
    \end{equation}
    and for $u \in \{61,109\}$ and $v \in \{41,89\}$, we have 
    \begin{equation} \label{thm2_}
         b_{4,5}^2\left( 8\times 5^{2j+1} n + \frac{u\times 5^{2j} -4}{3}\right) q^n \equiv0 \pmod{20},
    \end{equation}
    and 
     \begin{equation} \label{thm2__}
        b_{4,5}^2\left( 8\times 5^{2j+2} n + \frac{v\times 5^{2j+1} -4}{3}\right) q^n \equiv0 \pmod{20}.
    \end{equation}
\end{thm}
\begin{proof}
 From \eqref{tm4_1}, we have   
\begin{align}
\displaystyle \sum_{n=0}^\infty b^{2}_{4,5}(8n+3)q^{n}
&\equiv 2f_1^3 f_{10} \pmod{8}. \label{thm2_00}
\end{align}
The above equation proves \eqref{thm2} for $j=0$, Now suppose congruence \eqref{thm2} holds for some $j\geq 0$. Substituting cube of \eqref{f_1_dissection} in \eqref{thm2}, we get 
\begin{align}
    \displaystyle \sum_{n=0}^\infty b_{4,5}^2\left( 8\times 5^{2j} n + \frac{13\times 5^{2j} -4}{3}\right) q^n 
    &\equiv 2 f_{10} f_{25}^3 \nonumber \left( R(q^5)^{-1}-q-q^2 R(q^5)\right)^3 
    &\nonumber \\ 
    &\equiv 2 f_{10} f_{25}^3 \nonumber \\
    & \quad \times {\big[} R(q^5)^{-3}-3qR(q^5)^{-2} +5q^3 - 3 q^5 R(q^5)^2 -q^6 R(q^5)^3 {\big]} \nonumber \\
    & \quad \pmod{8}. \label{thm2_3} 
\end{align}
Extracting $q^{5n+3}$ from \eqref{thm2_3}, we get 
\begin{align}
\displaystyle \sum_{n=0}^\infty b_{4,5}^2\left( 8\times 5^{2j+1}n + \frac{17\times 5^{2j+1} -4}{3}\right) q^{n} 
    & \equiv 2 f_2 f_5^3 \pmod{8}. \label{thm2_4}
\end{align}
Invoking \eqref{f_1_dissection} in \eqref{thm2_4}, we get
\begin{align}
    \displaystyle \sum_{n=0}^\infty b_{4,5}^2\left( 8\times 5^{2j+1}n + \frac{17\times 5^{2j+1} -4}{3}\right) q^{n} 
    & \equiv 2 f_5^3 f_{50}\left( R(q^{10})^{-1}-q^2-q^4 R(q^{10})\right) \pmod{8}. \label{thm2_5}
\end{align}
Again, extracting the terms containing $q^{5n+2}$ from both sides of \eqref{thm2_5}, we get
\begin{align}
   \displaystyle \sum_{n=0}^\infty b_{4,5}^2\left( 8\times 5^{2j+2}n + \frac{13\times 5^{2j+2} -4}{3}\right) q^{n} 
    & \equiv 2 f_1^3 f_{10}\pmod{8}. \label{thm2_6}
\end{align}
So we obtain \eqref{thm2} with $j$ replaced by $j+1$. Extracting  $q^{5n+2},q^{5n+4}$ from \eqref{thm2_3} and \eqref{thm2_5} has no $q^{5n+1},q^{5n+3}$ terms and from comparing the arising equations with equation \eqref{b2_45_2n1}, we arrive at \eqref{thm2_} and \eqref{thm2__} respectively.
\end{proof}
\begin{thm}
If $n$ cannot be represented as the sum of a triangular number and ten times a pentagonal number, then $b^{2}_{4,5}(8n+3)\equiv 0 \pmod{8}.$
\end{thm}
\begin{proof}
    Extracting even terms on both sides of the equation \eqref{tm4_1} and using the binomial theorem, we get
    \begin{align}
\displaystyle \sum_{n=0}^\infty b^{2}_{4,5}(8n+3)q^{n} 
    \equiv 2 \psi(q) f_{10} \pmod{8}.\label{tm4_12}
    \end{align}
    \begin{align}
\displaystyle \sum_{n=0}^\infty b^{2}_{4,5}(8n+3)q^{n} 
    \equiv \sum_{l,k=-\infty}^{\infty}q^{\frac{l(l+1)}{2}+\frac{10k(3k+1)}{2}}\pmod{8}. \label{tm4_14}
    \end{align}
    The result follows.
\end{proof}

\begin{thm} For prime $p\geq 5$ , $\left(\frac{-30}{p}\right)=-1$ and $i=1,2,\dots,p-1$, we have
    \begin{equation}
         b_{4,5}^2\left( 8\times p^{2j+2}n + \frac{(24i+13p)\times p^{2j+1} -4}{3}\right) q^{n} 
        \equiv 0 \pmod{20}. \label{thm4_}
    \end{equation}
\end{thm}
\begin{proof}
 Since the proof of equation \eqref{thm4_} follows a similar argument as that of \eqref{cks1}, we omit the details.
 
 \end{proof}
 
 \begin{thm}
 	For any positive integer $n$, we have
 	\begin{align}
 		b^{2}_{8,5}(4n+1) \equiv 2 b^{2}_{2,5}(n) \pmod{10}. \label{tm10}
 	\end{align}
 \end{thm}
 \begin{proof}
 	Substituting $t=2$ in \eqref{tm5_2} and using the binomial theorem, we have
 	\begin{align}
 		\displaystyle \sum_{n=0}^\infty  b^{2}_{8,5}(2n+1)q^n &\equiv
 		2\frac{f_{2}^8}{f_{4}^8}\pmod{10}. \label{tm5_21}
 	\end{align}
 Extracting even powers of $q$ on both sides of the above equation \eqref{tm5_21}, we get
 	\begin{align}
 	\displaystyle \sum_{n=0}^\infty  b^{2}_{8,5}(4n+1)q^n &\equiv
 	2\frac{f_{1}^{10}f_2^2}{f_{2}^{10}f_1^2}\pmod{10}. \label{tm5_22}
 \end{align}
 Applying binomial theorem $ 2f_5 \equiv 2 f_1^5 \pmod{10} $ and $f^5_{j} \equiv 2 f_{5j} \pmod{10}$, we get 
 \begin{align}
 	\displaystyle \sum_{n=0}^\infty  b^{2}_{8,5}(4n+1)q^n &\equiv
 	2\frac{f_{5}^{2}f_2^2}{f_{10}^{2}f_1^2}\pmod{10}. \label{tm5_23}\nonumber\\ 
 	&\equiv 2\displaystyle\sum_{n=0}^{\infty} b^2_{2,5}(n)q^n\pmod{10}.
\end{align}
This completes the proof.
\end{proof}
\section{Congruences for $b^{3}_{2,3t}$}\label{s6}
\begin{thm}
For any nonnegative integer $n$ and $t \nmid 2$, we have 
    \begin{equation}\begin{split}
        b^{3}_{2,3t}(3n+2)\equiv 0 \pmod{6}.\label{tm7}        
    	\end{split}
    \end{equation}
\end{thm}

\begin{proof}
    Substituting $\ell=2$, $m=3t$ and $k=3$ in \eqref{main}, we have 
    \begin{equation}
        \displaystyle\sum_{n=0}^\infty b^{3}_{2,3t}(n)q^n= \frac{f_2^3 f_{3t}^3}{f_1^3 f_{6t}^3}.
    \end{equation}
    Substituting \eqref{f2.3_f1.3} in the above equation, we have
    \begin{align}
        \displaystyle\sum_{n=0}^\infty b^{3}_{2,3t}(n)q^n &= 
         \frac{f_6 f_{3t}^3}{f_3 f_{6t}^3}+3q
       \frac{f_6^4 f_9^5 f_{3t}^3 }{f_3^8 f_{18} f_{6t}^3}
        +6q^2\frac{f_6^2 f_9^2 f_{18}^2 f_{3t}^3}{f_3^7 f_{6t}^3} 
        +12q^3\frac{f_6^2 f_{18}^5 f_{3t}^3}{f_3^6 f_9 f_{6t}^3}. \label{tm7_1}
    \end{align}
Extracting terms containing $q^{3n+2}$ from \eqref{tm7_1}, we get 
    \begin{align}
        \displaystyle\sum_{n=0}^\infty b^{3}_{2,3t}(3n+2)q^{3n+2} &= 
        6q^2\frac{f_6^2 f_9^2 f_{18}^2 f_{3t}^3}{f_3^7 f_{6t}^3} .\label{b323t}
    \end{align}
    From the above equation, we arrive at \eqref{tm7}.
\end{proof}

\begin{thm} For any positive integer $n$, we have
     \begin{align}
         b^{3}_{2,3}(6n+5)&\equiv 0 \pmod{12}.\label{tm8}\\
          b^{3}_{2,3}(6n+2)q^n &\equiv 6b^{3}_3(n) \pmod{12}.\label{tm88}
          \end{align}
\end{thm}

\begin{proof}
Substituting $t=1$ in \eqref{b323t}, we get
    \begin{align}
        \displaystyle\sum_{n=0}^\infty  b^{3}_{2,3}(3n+2)q^{3n+2} &= 
        6q^2\frac{f_6^2 f_9^2 f_{18}^2 f_{3}^3}{f_3^7 f_{6}^3}. \label{tm8_1}
    \end{align}
Applying binomial theorem $6f_s^2 \equiv 6 f_{2s}\pmod{12}$ in \eqref{tm8_1}, we have
\begin{align}
        \displaystyle\sum_{n=0}^\infty  b^{3}_{2,3}(3n+2)q^n 
        &\equiv 6\frac{  f_{6}^3} { f_2^3} \pmod{12}.
\end{align}
Extracting the odd and even powers of $q$ on both sides of the above equation respectively, we arrive at \eqref{tm8} and \eqref{tm88} respectively, which completes the proof.
\end{proof}



{}
\end{document}